\documentclass[12pt]{amsart}
\usepackage{graphicx}
\usepackage{amsmath}
\usepackage{amsfonts}
\usepackage{amssymb}

\usepackage{amsxtra}
\usepackage{amssymb,latexsym}
\usepackage[mathcal]{eucal}
\usepackage{amscd}

\newtheorem{lemma}{Lemma}
\newtheorem{theorem}{Theorem}
\newtheorem{proposition}{Proposition}

\newtheorem{corollary}{Corollary}

\newcommand{\XBT}{(X, \mathcal{B}, \mu, T)}

\newcommand{\XBG}{(X, \mathcal{B}, \mu, T_g)}
\newcommand{\YCG}{(Y, \mathcal{C}, \nu, S_g)}
\newcommand{\eps}{\epsilon}
\newcommand{\emp}{\emptyset}
\newcommand{\bac}{\setminus}

\begin{document}

\bibliographystyle{plain}

\title{On the Jewett-Krieger theorem for amenable groups }
\author{ Benjamin Weiss}

\address {Institute of Mathematics\\
 Hebrew University of Jerusalem\\
Jerusalem\\
 Israel}
\email{benjamin.weiss1@mail.huji.ac.il}

\setcounter{page}{1}
\date{}

\maketitle
\begin{abstract}
   Up to now there has been no proof in the literature of the often quoted fact that the Jewett-Krieger theorem is valid for all
   countable amenable groups. In this brief note I will close this gap by applying a recent result of B. Frej and D. Huczek
   \cite{FH}.
\end{abstract}

\vspace{1cm}

  In his paper \cite{R} Alain Rosenthal showed that any free ergodic action of a countable amenable group $G$ with finite entropy equal to $h$
  has a \textbf{uniform generating partition} with $k$ elements if $h < \log(k-2)$. This is the same as saying  that it can be represented as a strictly ergodic subshift of $\{1, 2, \cdots k\}^G$.
  His proof was based on his preprint
 ``Strictly ergodic  models and amenable group actions". In this preprint
  Alain, generalizing the Jewett-Krieger theorem for $G=\mathbb{Z}$,  proved that every free ergodic action of a countable amenable group has a strictly ergodic model.
   Unfortunately this preprint has remained unpublished.
  This has left a gap in the literature.

  It turns out that a recent (2018) paper of Bartosz Frej and Dawid Huczek \cite{FH} essentially gives a proof of this result. Their main theorem
  is the following:
  \begin{theorem}
  Let $X$ be a 0-dimensional compact space  with a free topological action of an amenable group $G$ by homeomorphisms and
let $K$ be any  face of the simplex $M_G(X)$ of G-invariant measures of X. Then there exists
another free action of $G$ on a 0-dimensional compact space, $Y$ and an affine homeomorphism $\theta$ of $K$ onto the full simplex $M_G(Y)$,
such that for all $\mu \in K$ the systems $(X, \mu, G)$ and $(Y, \theta(\mu), G)$ are isomorphic.
  \end{theorem}

 As a  special case of their result,
  where the \textbf{face} of the simplex of invariant measures is simply a single ergodic measure, we have the  following.
  \begin{corollary}If $G$ is an amenable group, and the free ergodic action $\XBT$
  has an isomorphic representation $(Y, \nu, S_g)$,
  where $Y$ is a 0-dimensional space, the $S_g$'s are homeomorphisms, and \textbf{all} orbits are free,
    then there is a uniquely ergodic model for $\XBG$. This can be made strictly ergodic by simply restricting to the closed support of the measure.
   \end{corollary}

   All that remains to obtain a proof of the Jewett-Krieger theorem for amenable groups is to show that there always exists
  such a free topological model. This follows immediately from a general result of G\'{a}bor Elek \cite{E} which gives much more than is needed here. Thus it is
    perhaps worthwhile to give a direct proof of this fact.

     \begin{proposition}\label{free} Every free measure preserving action $\XBG$ of a countable group $G$ has a topologically free model on a 0-dimensional compact space.
  \end{proposition}

    The combination of the corollary and Proposition \ref{free} prove the following theorem which is a slight strengthening of the Jewett-Krieger
    theorem for amenable groups:

    \begin{theorem}  Every free ergodic measure preserving action $\XBG$ of a countable amenable group $G$ has a strictly ergodic model $(Y, G)$ where
    the action is topologically free, i.e. all orbits are free.
    \end{theorem}

    We turn to the proof of the proposition in which the group $G$ need not be amenable and only the freeness of the measure preserving action is assumed.
    As a first step we will see that we may assume that the space $X$ is compact, 0-dimensional and that the action is by homeomorphisms.
    Indeed since by the fact that all standard measure spaces are isomorphic, we can assume that the space $X$ in our measure preserving system is the Cantor set $C$.
    Set $Y = C^G$ with the shift action and define $\phi: X \rightarrow Y$ by the formula $\phi(x)(g) = T_g(x)$. This mapping is one to one and equivariant and therefore the image of $\mu$
     under the mapping $\phi$ and the measure $\phi \circ \mu$ gives us a system isomorphic $\XBG$. We will retain the notation $\XBG$ for this new system. We will now construct a factor of $\XBG$ which is
     topologically free and then the factor joining (which we will define below) will give us the desired result. For this the main tool will
     be the following lemma.

      \begin{lemma}\label{basic} Let $(X, T_g)$ be a continuous action of a countable group $G$ on a 0-dimensional compact metric space $(X, d)$, and let  $\nu$ be a $G$-invariant
   measure such that for all $g \neq e$ the set $E_g = \{ x \in X | T_gx = x \}$ satisfies  $\mu(E_g) = 0$. Then, for each $g \neq e$,  there is an open set $A$ that satisfies: \\

     (i) $A \cap T_gA = \emptyset.$ \\

     (ii) $\mu(T_{g^{-1}}A \cup A \cup T_gA) = 1.$ \\

    (iii) The boundary of $A$ is contained in $E_g.$
   \end{lemma}
   \begin{proof}In the following we fix a $g \neq e$ and will denote $E_g$ simply by $E$. We will also use the notation $E_{\eps}$ for
   the set $\{ x \in X | d(x, E) < \eps \}$.
   The set $A$ will be constructed as the  union of a sequence of clopen sets $A_i$.

   For the first step choose $\eps_1$ so that $\mu(E_{\eps_1}) < \frac{1}{10}$ and set $Z_1 = X \bac E_{\eps_1}$. Since $T_g$ is continuous, for each
   $z \in Z_1$ there is a clopen neighborhood of $z$, $U_z$,  such that $U_z \cap T_gU_z = \emp$. Since $Z_1$ is closed there is a finite collection of
   clopen sets $\{U_i | 1 \leq i \leq k\}$, that for each $i$ satisfy \\
   $U_i \cap T_gU_i = \emp$
  and $\bigcup_{i=1}^k U_i = Z_1$. Set $B_1 = U_1$ and then define

   $$ B_2 = B_1 \cup   (U_2 \bac (T_{g^{-1}}B_1 \cup B_1 \cup T_gB_1).$$
    It follows that $B_2 \cup T_gB_2 = \emp$.
   Continue in this fashion, namely for $i < k$ set define

   $$B_{i+1} = B_i \cup   (U_{i+1} \bac (T_{g^{-1}}B_i \cup B_i \cup T_gB_i).$$
     and set $A_1 = B_k$. This is a clopen set such that $C_1=T_{g^{-1}}A_1 \cup A_1 \cup T_gA_1$ covers $Z_1$ and satisfies the condition $A_1 \cap T_gA_1 = \emp$. Observe that $C_1$ is clopen and
    disjoint from $E$ and we can therefore find an $\eps_2$ such that it is also disjoint from $E_{\eps_2}$ and in addition $\mu(E_{\eps_2}) < \frac{1}{100}$.
     Define
      $Z_2 = X \bac (E_{\eps_2} \cup C_1)$
   and observe that it is closed and disjoint from $E$.

  We will now repeat the construction that was done for Z$_1$.  Notice that for $z \in Z_2$, since $C_1$ is clopen, the clopen neighborhood $U_z$ of $z$
     for which $U_z \cap T_gU_z = \emp$
     can be taken small enough so that $T_{g^{-1}}U \cup U \cup T_gU$ is disjoint from $C_1$. As before finitely many such neighborhoods will cover $Z_2$
    and with them we construct $A_2$ so that $C_2 = T_{g^{-1}}A_2 \cup A_2 \cup T_gA_2$ covers $Z_2$ and $(A_1 \cup A_2) \cap T_g (A_1 \cup A_2)=\emp$.

    The general step should be clear, and the construction yields a sequence  $A_i$ so that their union $A = \bigcup_1A_i$ satisfies the conclusions of the lemma.
    $\Box$ \\
 \end{proof}

       We return now to the proof of the proposition.
       \begin{proof}
      Fix an enumeration of the elements of the group $G$,\\
       $\{e, g_1, g_2, \cdots g_n, \cdots \}$, and for each $g_i$ let $A_i$ be a subset satisfying the conclusion
  of Lemma \ref{basic} with $g=g_i$. For each $i$ define
  $f_i:X \rightarrow \{0,1 \}^G$ by $f_i(x)[h] = \mathbb{I}_{A_i}( T_hx)$. This mapping is equivariant
  with respect to the shift action of $G$ on  $\{0,1 \}^G$.  Now set
  $\Omega = \{ 0, 1 \}^{\mathbb{N} \times G}$ and define $F: X \rightarrow \Omega$ by $F(x)(i, h) = f_i(x)[h]$. This mapping is equivariant with respect to the diagonal action of the
  shift on the countable product of $\{0,1 \}^G$. If $\Omega_0$ is the topological support of $F \circ \mu$, it is clear that the action of $G$ on $\Omega_0$
   is topologically free.

   Consider the product $X \times \Omega_0$, with the diagonal action
of $G$. Obviously, this system is topologically free, because it has the free factor
$\Omega_0$. On the other hand, the mapping $\hat F:X \to X \times \Omega_0$
given by $\hat F(x) = (x, F(x))$ is a measure isomorphisms between $\mu$ on $X$
and $\hat F\circ \mu$ on $X \times \Omega_0$. So, $X \times \Omega_0$ is the
desired topologically free model of $(X,\mathcal B,\mu,T_g)$.

  $\Box$
       \end{proof}

   In conclusion I would like to point out that the preprint of A. Rosenthal mentioned above also contains a proof of a relative version of the Jewett-Krieger theorem for amenable groups.
   This concerns the situation where we  have a free ergodic system $\XBG$ and a free factor system $\YCG$ given by a mapping $\pi:X \rightarrow Y$ such that $\pi^{-1}(\mathcal{C}) = \mathcal{B}$, and
   $\pi \circ \mu = \nu$ and $\pi T_g = S_g \pi$ for all $g \in G$. We are given in addition a strictly ergodic model $(\hat{Y}, G)$ of $\YCG$, and the relative Jewett-Krieger theorem asserts that there is a topological
   extension $(\hat{X}, G)$ of $(\hat{Y}, G)$ given by a continuous equivariant map $\hat{\pi}: \hat{X} \rightarrow \hat{Y}$ such that $(\hat{X}, G)$
   is a strictly ergodic model of $\XBG$. For a more detailed discussion of this kind of generalization see \cite{W}.
    Readers who are interested in the preprint should contact me.

    I would like to thank Tomasz Downarowicz for a careful reading of earlier versions of this note.


\begin{thebibliography}{WWW}

\bibitem[E]{E}
Elek, G\'{a}bor,
    {\em Free minimal actions of countable groups with invariant
              probability measures},
   Ergodic Theory Dynam. Systems, {\bf{41}}, 2021, pp. 1369-1389.


\bibitem[FH]{FH}
Frej, Bartosz and Huczek, Dawid,
 {\em Faces of simplices of invariant measures for actions of
              amenable groups},
Monatsh. Math.,  {\bf{185}}, 2018, pp. 61-80.

\bibitem[R]{R}
Rosenthal, Alain,
      {\em Finite uniform generators for ergodic, finite entropy, free
              actions of amenable groups},
   Probab. Theory Related Fields, {\bf{77}} , 1988, pp. 147-166.

\bibitem[W]{W}
 Weiss, Benjamin,
     {\em Strictly ergodic models for dynamical systems},
   Bull. Amer. Math. Soc. (N.S.),
  American Mathematical Society. Bulletin. New Series,
    {\bf{13}}, 1985, pp. 143-146.


\end{thebibliography}
\end{document}